\documentclass[journal]{IEEEtran}

\IEEEoverridecommandlockouts                              

\usepackage{changes}
\usepackage{authblk}
\usepackage{amsfonts}
\usepackage{amsthm}
\usepackage{amsmath,amssymb}       
\usepackage{hyperref,cite}
\usepackage{cleveref}
\usepackage{color}
\usepackage{algorithm}
\usepackage{algorithmic}
\usepackage{makecell}
\usepackage{tabularx}
\let\labelindent\relax
\usepackage{enumitem}
\usepackage{float}
\usepackage{graphicx}
\usepackage{booktabs}
\usepackage{threeparttable}
\usepackage{hhline}
\usepackage{empheq}

\newtheorem{theorem}{Theorem}

\newtheorem{remark}{Remark}
\newtheorem{lemma}{Lemma}
\DeclareMathOperator*{\argmin}{arg\,min} 
\newcommand{\diag}{\mathop{\rm diag}\nolimits}

\newcommand{\ba}[1]{\begin{array}{#1}}
\newcommand{\ea}{\end{array}}

\begin{document}
\title{\LARGE \bf
$\pi$MPC: A Parallel-in-horizon and Construction-free NMPC Solver
}
\author{Liang Wu$^{1,*}$, Bo Yang$^{2,*}$, Junheng Li$^{3,*}$, Xu Yang$^{2}$, Yilin Mo$^{2}$, Yang Shi$^{4}$, Aaron D. Ames$^{3}$, and J\'an Drgo\v na$^{1}$%
\thanks{$^*$Equal Contributions. $^{1}$Johns Hopkins University, MD 21218, USA. $^{2}$Tsinghua University, Beijing 100084, China. $^{3}$California Institute of Technology, CA 91106, USA, $^{4}$University of Victoria, Victoria, BC V8N 3P6, Canada. Corresponding author: {\tt\small wliang14@jh.edu}.\\
This research is in part supported by the Ralph O’Connor Sustainable Energy Institute at Johns Hopkins University and the Technology Innovation Institute.
}
}

\maketitle

\begin{abstract}
The alternating direction method of multipliers (ADMM) has gained increasing popularity in embedded model predictive control (MPC) due to its code simplicity and pain-free parameter selection. However, existing ADMM solvers either target general quadratic programming (QP) problems or exploit sparse MPC formulations via Riccati recursions, which are inherently sequential and therefore difficult to parallelize for long prediction horizons. This technical note proposes a novel \textit{parallel-in-horizon} and \textit{construction-free} nonlinear MPC algorithm, termed $\pi$MPC, which combines a new variable-splitting scheme with a velocity-based system representation in the ADMM framework, enabling horizon-wise parallel execution while operating directly on system matrices without explicit MPC-to-QP construction. Numerical experiments and accompanying code are provided to validate the effectiveness of the proposed method.
\end{abstract}


\section{Introduction}
Model predictive control (MPC) is a widely used model-based optimal control framework in manufacturing, energy systems, and robotics, where control actions are obtained by solving an online optimization problem, typically formulated as a quadratic program (QP), based on a prediction model, constraints, and an objective function. Over the past decades, the MPC community has developed a wide range of fast QP algorithms, spanning interior-point methods (IPMs) \cite{Wang2010FastMP, wu2025direct,wu2025eiqp,wu2025quadratic}, active-set methods (ASMs) \cite{ferreau2014qpoases,ferreau2008online,bemporad2015quadratic}, and first-order methods (FOMs) such as gradient projection \cite{patrinos2013accelerated}, the alternating direction method of multipliers (ADMM) \cite{boyd2011distributed, SBGBB20}), and and other techniques \cite{hermans2019qpalm, wu2023simple, wu2023construction}, for deployment on embedded controller platforms.

Code simplicity is a critical requirement for industrial embedded MPC algorithms, as it enables straightforward verification, validation, and maintenance on embedded controller platforms. As a result, FOMs have gained increasing attention, since they rely solely on matrix–vector operations and projections, without requiring linear system solves, in contrast to IPMs and ASMs. Among FOMs for MPC problems, ADMM stands out because it offers pain-free parameter selection. Unlike many gradient-based methods that require estimating the Lipschitz constant to choose the step size, ADMM converges for any $\rho>0$. This advantage is particularly significant in linear time-varying or real-time iteration (RTI) nonlinear MPC (NMPC) settings \cite{gros2020linear}, where the Hessian matrix of the underlying MPC-QP is time-varying; in such cases, gradient-based methods require careful step-size calculation, thereby increasing the complexity of NMPC implementations, whereas ADMM does not.

\subsection{Related work}
ADMM is well known for its problem-decomposition properties, which naturally support constraint handling as well as parallel and distributed computation. For example, \cite{SBGBB20} proposes an ADMM-based solver for general constrained QP problems, in which a variable-splitting scheme facilitates constraint handling. Moreover, \cite{tang2019distributed, wang2017distributed, rey2020admm} employ ADMM to design distributed MPC schemes for interconnected systems or systems with local (uncoupled) and global (coupled) constraints. In contrast, this technical note focuses on developing ADMM algorithms tailored specifically for MPC applications, without assuming that the controlled system possesses special structural properties.

ADMM algorithms typically perform more efficiently on sparse MPC-QP formulations than on condensed MPC-QP formulations \cite{jerez2011condensed}, since the latter eliminate the system states and involve only control inputs, which often leads to poorer conditioning. When ADMM is applied to sparse MPC-QP formulations, exploiting the structure of linear state-space dynamics, e.g., via Riccati recursions, \cite{annergren2012admm,sokoler2014input}, can not only reduce the computational complexity to scale linearly with the prediction horizon but also enable a \textit{construction-free} implementation. In particular, no explicit MPC-to-QP construction is required, as the algorithm operates directly on the system matrices $\{A_{t,k}, B_{t,k}\}_{k=0}^{N-1}$ rather than assembling a large sparse matrix over the prediction horizon $N$. The \textit{construction-free} feature reduces code complexity in NMPC implementations (e.g., the online linearized or the RTI scheme) and eliminates online MPC-to-QP overhead. However, Riccati recursion is inherently sequential and thus hard to parallelize, which can lead to slow computation for long-horizon MPC problems. Consequently, there is a strong need for \textit{construction-free} ADMM iterations that also achieve \textbf{linear- and parallel-in-horizon} cost.

\subsection{Contribution}
This technical note proposes a novel \textit{\underline{\textbf{p}}arallel-\underline{\textbf{i}}n-horizon} and \textit{construction-free} NMPC algorithm, referred to as $\pi$MPC. The ADMM-based $\pi$MPC algorithm achieves \textit{construction-free} and \textit{parallel-in-horizon} execution through a new variable-splitting scheme (introducing copies of $\{x_{k+1}\}$ and $\{B_{t,k}u_k\}$) together with a velocity-based formulation, while retaining closed-form iteration updates. In particular, a convergence-guaranteed accelerated ADMM framework incorporating a restart scheme is adopted to enhance computational efficiency.

Moreover, owing to its construction-free design, $\pi$MPC features low code complexity in NMPC implementations and a reduced deployment barrier on embedded platforms, while its parallel computational efficiency makes it particularly attractive for long-horizon NMPC problems.

\section{Problem Formulation}
Consider an MPC problem for a linear-time-varying system, which can be viewed as a generalized formulation from online-linearized- or RTI-NMPC (see \cite{gros2020linear}), as follows,
\begin{subequations}\label{eqn_MPC}
    \begin{align}
                \min_{U,X}&~ \sum_{k=0}^{N-1}\frac{1}{2}\|Cx_{k+1}-r_y\|_{W_y}^2 + \frac{1}{2}\|u_k- r_u\|_{W_u}^2\label{eqn_MPC_a}\\
        \text{s.t.}&~ x_{k+1} = A_{t,k}x_{k} + B_{t,k}u_k + e_{t,k},~\forall k\in\mathbb{N}_{0}^{N-1}\label{eqn_MPC_b}\\
        &~ u_{k}\in\mathcal{U},~x_{k+1}\in\mathcal{X}, ~\forall k\in\mathbb{N}_{0}^{N-1}\label{eqn_MPC_c}\\
        &~ x_0 = x(t),\label{eqn_MPC_d}
    \end{align}
\end{subequations}
where $N$ denotes the length of the prediction horizon; $U=\mathrm{col}(u_0,\cdots,u_{N-1})$ and $X=\mathrm{col}(x_1,\cdots,x_N)$ denote the control input and state sequence along the prediction horizon, respectively; the dynamic-related matrices and affine terms $\{A_{t,k},B_{t,k},e_{t,k}\}_{k=0}^{N-1}$ can be time-varying across the prediction horizon; $C$ denotes the output matrix, $r_y$ and $r_u$ denote the tracking references; $W_y$ and $W_u$ denote the weights symmetric matrices; $\mathcal{X}$ and $\mathcal{U}$ denote closed convex constraint sets.
\begin{remark}\label{remark_1}
In this article, we assume that Problem \eqref{eqn_MPC} is convex and feasible. If $W_y,~W_u$ are positive semi-definite and $\mathcal{X},~\mathcal{U}$ are convex, then Problem \eqref{eqn_MPC} is convex; Under the current feedback state $x(t)$ (Eqn. \eqref{eqn_MPC_d}), if there exist trajectories $(U,X)$ that satisfy Eqn. \eqref{eqn_MPC_b} and \eqref{eqn_MPC_c}, then Problem \eqref{eqn_MPC} is feasible.
\end{remark}

\section{Preliminaries: Previous ADMM for MPC}
\subsection{General ADMM algorithm}
Generally, Problem \eqref{eqn_MPC} can be reformulated as a sparse general QP formulation as follows
\begin{equation}\label{eqn_QP}
        \min_y~ \frac{1}{2}y^\top H_t y + h_t^\top y,~\text{s.t}~ E_t y\in\mathcal{Y}_t
\end{equation}
where $y$ can be $\mathrm{col}(u_0,x_1,\cdots,u_{N-1},x_{N})$, $H_t,~h_t$ denote the the assembly of Eqn. \eqref{eqn_MPC_a} (time-varying if the weight matrices and reference signals are time-varying), $E_t,~\mathcal{Y}_t$ denote the assembly of Eqns. \eqref{eqn_MPC_b} and \eqref{eqn_MPC_c} (note that Eqn. \eqref{eqn_MPC_b} constitutes an equality constraint, which can be equivalently reformulated as a pair of inequality constraints). It should be noted that since $x(t),\{A_{t,k},B_{t,k},e_{t,k}\}_{k=0}^{N-1}$ are potentially time-varying, the resulting $E(t)$ and $\mathcal{Y}(t)$ exhibit time dependence; this notation is adopted to reflect such temporal variation explicitly.

A simple splitting scheme introduces $w=E_ty$, making the QP \eqref{eqn_QP} equivalent to the following formulation:
\begin{equation}
        \min_{y,w}~ \frac{1}{2}y^\top H_t y + h_t^\top y + \Pi_{\mathcal{Y}_t}(w),~\text{s.t}~ E_ty-w=0 
\end{equation}
where $\Pi_{\mathcal{Y}_t}(w)$ denotes the indicator function of the constraint set $\mathcal{Y}_t$, and its augmented Lagrangian (with the scaled-form Lagrangian variable) is given by $L_{\rho}(y,w,\lambda) = \frac{1}{2}y^\top H_t y + h_t^\top y + \Pi_{\mathcal{Y}_t}(w) + \frac{\rho}{2}\|E_ty-w+\lambda\|_2^2$. Based on \cite{boyd2011distributed}, the corresponding $(i+1)$-th ADMM update is
\begin{subequations}\label{eqn_ADMM_QP_iter}
{\small
\begin{empheq}[left=\empheqlbrace]{align}
    y^{i+1}&=\argmin_{y}\frac{1}{2}y^\top H_t y + y^\top h_t+\frac{\rho}{2}\|E_ty-w^i+\lambda^i\|_2^2\label{eqn_ADMM_QP_iter_a}\\
    &=(H_t+\rho E_t^\top E_t)^{-1}\Big(\rho E_t^\top (w^i-\lambda^i)-h_t\Big)\nonumber\\
    w^{i+1} &=  \argmin_{w}\Pi_{\mathcal{Y}_t}(w)  +\frac{\rho}{2}\|E_ty^{i+1}-w+\lambda^i\|_2^2\\
    &=\mathrm{Proj}_{\mathcal{Y}_t}\left(E_ty^{i+1}+\lambda^i \right) \nonumber\\
    \lambda^{i+1}&=\lambda^{i} +  E_ty^{i+1}-w^{i+1}
    \end{empheq}
}
\end{subequations}
\begin{remark}
If $H_t, E_t$ from the MPC-to-QP construction are time-invariant and the value of $\rho$ is fixed during iterations, $(H(t) +\rho E(t)^\top E(t))^{-1}$ can be computed offline and stored for use during the iterations. Otherwise, such as in NMPC problems, the computation cost of the matrix inverse in Eqn. \eqref{eqn_ADMM_QP_iter_a} is \textbf{cubic} in the prediction horizon and not parallel-in-horizon.
\end{remark}

\subsection{Structure-exploited ADMM via Riccati Recursion}
To reduce the possible cubic-in-horizon cost, some structure-exploited ADMM algorithms tailored for MPC adopt the use of the well-known Riccati Recursion, see \cite{annergren2012admm,sokoler2014input}. By decoupling the inequality constraints \eqref{eqn_MPC_c} and the dynamic equality constraints \eqref{eqn_MPC_b} by introducing auxiliary copies $(V, Z)$ of $(U, X)$, as follows,
\begin{equation}\label{eqn_MPC_ADMM_1}
        \min_{U,X,V,Z}~ f(U,X) + g(V,Z),~\text{s.t. } U = V,~ X = Z
\end{equation}
where $V\triangleq\mathrm{col}(v_0,\cdots,v_{N-1})$ and $Z\triangleq\mathrm{col}(z_1,\cdots,z_N)$; the two terms of objective function are given by
\[
\begin{aligned}
    f(U,X) \triangleq \sum_{k=0}^{N-1}&\frac{1}{2}\|Cx_{k+1}-r_y\|_{W_y}^2 + \frac{1}{2}\|u_k- r_u\|_{W_u}^2\\
    +~& \text{Dyn}_{x_0=x(t)}(u_k,x_k,x_{k+1})    \\
    g(V,Z)\triangleq \sum_{k=0}^{N-1}& \Pi_{\mathcal{U}}(v_k) + \Pi_{\mathcal{X}}(z_{k+1})
\end{aligned}
\]
with $\text{Dyn}_{x_0=x(t)}(u_k,x_k,x_{k+1})$ is the indictor function of $x_{k+1}-A_{t,k}x_{k} - B_{t,k}u_k - e_{t,k} = 0$; $\Pi_{\mathcal{U}}$ and $\Pi_{\mathcal{X}}$ are the indictor function of constraint sets $\mathcal{U}$ and $\mathcal{X}$, respectively. Then, the augmented Lagrangian (with the scaled-form Lagrangian variable) for Problem \eqref{eqn_MPC_ADMM_1} is $ L_{\rho}(U,X,V,Z,\Phi,\Psi)\triangleq f(U,X) + g(V,Z) + \frac{\rho}{2}\|U-V+\Phi\|_2^2 + \frac{\rho}{2}\|X-Z+ \Psi\|_2^2$, where $\Phi\triangleq\mathrm{col}(\phi_0,\cdots,\phi_{N-1})$ and  $\Psi\triangleq\mathrm{col}(\psi_1,\cdots,\psi_{N})$ denotes the dual variable for equality constraints $u_k=v_k$ and $x_{k+1}=z_{k+1}$ for $k=0,\cdots,N-1$, respectively. Based on \cite{boyd2011distributed}, the corresponding $(i+1)$-th ADMM update is given by
\begin{subequations}
{\small
\begin{empheq}[left=\empheqlbrace]{align}
    &\left(\begin{array}{c}
         U^{i+1}  \\
         X^{i+1} 
    \end{array}\right)=\argmin_{U,X} f(U,X)+\frac{\rho}{2}\left\|U-V^{i}+\Phi^i\right\|_2^2 \label{eqn_ADMM_1_update}\\
         &\qquad\qquad\quad\quad\quad\quad + \frac{\rho}{2}\left\|X-Z^i+ \Psi^i\right\|_2^2\nonumber\\
    &\left(\begin{array}{c}
         V^{i+1}  \\
         Z^{i+1}
    \end{array}\right)=\argmin_{V,Z} g(V,Z)+\frac{\rho}{2}\left\|U^{i+1}-V+ \Phi^i\right\|_2^2\label{eqn_ADMM_2_update}\\
    &\qquad\qquad\quad\quad\quad\quad + \frac{\rho}{2}\left\|X^{i+1}-Z+\Psi^i\right\|_2^2 \nonumber\\
    &\left(\begin{array}{c}
         \Phi^{i+1}  \\
         \Psi^{i+1} 
    \end{array}\right) = \left(\begin{array}{c}
         \Phi^i + U^{i+1} - V^{i+1}  \\
         \Psi^i + X^{i+1} - Z^{i+1}
    \end{array}\right)\label{eqn_ADMM_3_update}
    \end{empheq}
}
\end{subequations}
Particularly, the update Eqn. \eqref{eqn_ADMM_1_update} is the solution of the following unconstrained MPC problem,
\[
    \begin{aligned}
        \min_{U,X}&~\sum_{k=0}^{N-1} \frac{1}{2}\|Cx_{k+1}-r_y\|_{W_y}^2 + \frac{1}{2}\|u_k- r_u\|_{W_u}^2\\
        &~\quad +\frac{\rho}{2}\left\|u_k-v_k^i+\phi_k^i\right\|_2^2 +\frac{\rho}{2}\left\|x_{k+1}-z_{k+1}^i+\psi_{k+1}^i\right\|_2^2\\
        \text{s.t.}&~ x_{k+1} = A_{t,k}x_{k} + B_{t,k}u_k + e_{t,k}, ~\forall k\in\mathbb{N}_{0}^{N-1}\\
         &~ x_0 = x(t)
    \end{aligned}
\]
which can be solved by the Riccati Recursion with \textbf{linear} complexity on the prediction horizon $N$ even when $\{A_{t,k},B_{t,k},e_{t,k}\}_{k=0}^{N-1}$ are time-varying and the value of $\rho$ is changing during ADMM iterations.

\begin{remark}
 However, the implementation of Riccati Recursion is inherently sequential and hard to parallelize. Although the update Eqn. \eqref{eqn_ADMM_2_update}, which involves simple projection operations, and the update Eqn. \eqref{eqn_ADMM_3_update} can be parallelized; its computational cost accounts for only a small fraction of the total cost. Consequently, the main computational bottleneck lies in the sequential operation Eqn. \eqref{eqn_ADMM_1_update}.   
\end{remark}

\section{Parallel-in-horizon ADMM Algorithm}
To achieve both \textbf{linear-} and \textbf{parallel-in-horizon} computation cost per ADMM iteration, this paper introduces the stacked variables $V=\mathrm{col}(v_0,\cdots,v_{N-1})$ and $Z=\mathrm{col}(z_1,\cdots,z_N)$ as follows,
\begin{equation}\label{eqn_MPC_ADMM_2}
\begin{aligned}
     \min_{U,X,V,Z}&~ \sum_{k=0}^{N-1}\frac{1}{2}\|Cx_{k+1}-r_y\|_{W_y}^2 + \frac{1}{2}\|u_k- r_u\|_{W_u}^2\\
     \text{s.t.}&~ x_{k+1} - z_{k+1}=0,~\forall k\in\mathbb{N}_{0}^{N-1}\\
    &~ B_{t,k}u_k - v_k =0,~\forall k\in\mathbb{N}_{0}^{N-1}\\
     &~ z_{k+1} - A_{t,k}x_k- v_k - e_{t,k}=0,~\forall k\in\mathbb{N}_{0}^{N-1}\\
     &~ u_k\in\mathcal{U},~z_{k+1}\in\mathcal{X}~\forall k\in\mathbb{N}_{0}^{N-1}
\end{aligned}
\end{equation}
\textbf{Note that the above splitting scheme is the key to achieving linear parallel-in-horizon and construction-free features, such as making a copy of $v_k=B_{t,k}u_k$ instead of $v_k=u_k$, and incorporating $z_{k+1}$, rather than $x_{k+1}$, into the dynamical equation.} By representing the inequality constraints $u_k\in\mathcal{U}, z_{k+1}\in\mathcal{X}$ for $k\in\mathbb{N}_{0}^{N-1}$ as the indicator functions and incorporating them into the objective, namely
\[
\begin{aligned}
    f(U,X)&\triangleq \sum_{k=0}^{N-1} \frac{1}{2}\|Cx_{k+1}\!-\!r_y\|_{W_y}^2 \!+\!\frac{1}{2} \|u_k\!-\!r_u\|_{W_u}^2 \!+\! \Pi_{\mathcal{U}}(u_k)\\
    g(V,Z)&\triangleq \sum_{k=0}^{N-1} \Pi_{\mathcal{X}}(z_{k+1}), 
\end{aligned}
\]
the augmented Lagrangian (with the scaled-form Lagrangian variable) is given by
\[
\begin{aligned}
    &L_{\rho}(U,X,V,Z,\Theta,\mathcal{B},\Lambda)=f(U,X) + g(V,Z)\\
    &+ \frac{\rho}{2}\sum_{k=0}^{N-1} \|x_{k+1}-z_{k+1}+\theta_k\|_2^2+ \frac{\rho}{2}\sum_{k=0}^{N-1} \|B_{t,k}u_k-v_k+\beta_k\|_2^2\\
    &+ \frac{\rho}{2}\sum_{k=0}^{N-1} \|z_{k+1}-A_{t,k}x_k-v_k-e_{t,k}+\lambda_k\|_2^2
\end{aligned}
\]
where $\Theta\triangleq\mathrm{col}(\theta_0,\cdots,\theta_{N-1})$, $\mathcal{B}\triangleq\mathrm{col}(\beta_0,\cdots,\beta_{N-1})$, and $\Lambda\triangleq\mathrm{col}(\lambda_0,\cdots,\lambda_{N-1})$. Then, the corresponding $(i+1)$-th ADMM update is given by
\begin{subequations}\label{eqn_ADMM2_iter}
{\small
\begin{empheq}[left=\empheqlbrace]{align}
&\textbf{for } k=0,\cdots, N-1~\textbf{(in parallel)} \label{eqn_ADMM2_u_x_update}\\
    &\quad u_{k}^{i+1}=\argmin_{u_k}~ \frac{1}{2}\|u_k- r_u\|_{W_u}^2+\Pi_{\mathcal{U}}(u_k)\nonumber\\
     &\quad\quad + ~\frac{\rho}{2}\left\|B_{t,k}u_k-v_k^i+\beta_k^i\right\|_2^2  \nonumber \\
     &\quad x_{k+1}^{i+1}=\argmin_{x_{k+1}}~ \frac{1}{2}\|Cx_{k+1}-r_y\|_{W_y}^2\nonumber\\
      &\quad\quad+\frac{\rho}{2}\|x_{k+1}-z_{k+1}^i+\theta_k^i\|_2^2 \nonumber\\
     &\quad\quad+ \frac{\rho}{2}\|z_{k+2}^i-A_{t,k}x_{k+1}-v_{k+1}^i-e_{t,k}+\lambda_{k+1}^i\|_2^2\nonumber\\
    &\textbf{for } k=0,\cdots, N-1~\textbf{(in parallel)}\label{eqn_ADMM2_v_z_update}\\
    &\quad(v_{k}^{i+1},z_{k+1}^{i+1})=\argmin_{v_k,z_{k+1}}~ \Pi_{\mathcal{X}}(z_{k+1})\nonumber\\
    &\quad + \frac{\rho}{2}\Big(\|x_{k+1}^{i+1}-z_{k+1}+\theta_k^i\|_2^2 + \|B_{t,k}u_{k}^{i+1}-v_k+\beta_k^i\|^2_2 \nonumber\\
    &\quad+\|z_{k+1}-A_{t,k}x_k^{i+1}-v_k-e_{t,k}+\lambda_k^i\|_2^2\Big)\nonumber\\ 
   &\textbf{for } k=0,\cdots, N-1~\textbf{(in parallel)}\label{eqn_ADMM2_theta_beta_lambda_update} \\
   &\quad\theta_k^{i+1}=\theta_k^i + x_{k+1}^{i+1}-z_{k+1}^{i+1} \nonumber\\
    &\quad\beta_k^{i+1}=\beta_k^{i} + B_{t,k}u_{k}^{i+1} - v_k^{i+1}\nonumber\\
    &\quad\lambda_k^{i+1}=\lambda_k^i + z_{k+1}^{i+1} -A_{t,k}x_{k}^{i+1} - v_{k+1}^{i+1}-e_{t,k}\nonumber
    \end{empheq}
}
\end{subequations}
Note that in Eqn. \eqref{eqn_ADMM2_u_x_update}, the update of primal variables$(U,X)$, when $k=N-1$ the term $\frac{\rho}{2}\|z_{k+2}^i-A_{t,k}x_{k+1}-v_{k+1}^i-e_{t,k}+\lambda_{k+1}^i\|_2^2$ does not appear, and \textbf{the updates of $u_k$ and $x_{k+1}$ are decoupled}; in Eqn. \eqref{eqn_ADMM2_theta_beta_lambda_update}, the update of the dual variables $\{\theta_k,\beta_k,\lambda_k\}_{k=0}^{N-1}$, when $k=0$ the term $x_k^{i+1}$ is $x(t)$ and \textbf{the updates of $\theta_k,\beta_k$, and $\lambda_k$ are decoupled}.

\begin{remark}
    Clearly, the new splitting scheme in Problem \eqref{eqn_MPC_ADMM_2} enables all ADMM updates to be executed in parallel over the prediction horizon, i.e., it is \textbf{parallel-in-horizon}, and the cost scales \textbf{linearly} with the horizon length.
\end{remark}

Although the \textbf{linear parallel-in-horizon} property is appealing, it \textbf{appears} to come at the cost of more complex update steps in Eqns. \eqref{eqn_ADMM2_u_x_update} and \eqref{eqn_ADMM2_v_z_update}, which, unlike the updates in Eqns. \eqref{eqn_ADMM_1_update} and \eqref{eqn_ADMM_2_update}, do not admit closed-form solutions.

However, this paper shows that this is not the case: the update in \eqref{eqn_ADMM2_v_z_update} can in fact be simplified to a closed-form expression, as established in Theorem~\ref{thm_v_z}. Furthermore, a velocity-based formulation is introduced in Subsection~\ref{sec_velocity_formulation}, enabling a closed-form solution for the update in \eqref{eqn_ADMM2_u_x_update} as well, as shown in Theorem~\ref{thm_2}.

\begin{theorem}\label{thm_v_z}
    The update Eqn. \eqref{eqn_ADMM2_v_z_update} admits a closed-form solution as follows,
    \begin{equation}\label{eqn_ADMM2_v_z_update_closed_form}
    \begin{aligned}
    &\textbf{for } k=0,\cdots, N-1~\textbf{(in parallel)}\\
    &\quad\gamma_k = B_{t,k}u_k^{i+1}+\beta_k^i-A_{t,k}x_k^{i+1}-e_{t,k}+\lambda_k^i\\
    &\quad z_{k+1}^{i+1}=\mathrm{Proj}_{\mathcal{X}}\left(\frac{2x_{k+1}^{i+1}+2\theta_k^i+\gamma_k}{3}\right)\\   
    &\quad v_{k}^{i+1}=\frac{1}{2}\Big(z_{k+1}^{i+1}+\gamma_k\Big)
    \end{aligned}
\end{equation}
\end{theorem}
\begin{proof}
Applying the Karush–Kuhn–Tucker (KKT) condition (see \cite[Ch 5]{boyd2004convex}) on Eqn. \eqref{eqn_ADMM2_v_z_update}, the KKT equation related to the  unconstrained $v_k$ is
\[
   \rho\Big(2v_k^*-(B_{t,k}u_k^{i+1}+\beta_k^i)+(A_{t,k}x_k^{i+1}+e_{t,k}-\lambda_k^i)-z_{k+1}^*\Big)=0,
\]
For simplicity, we introduce the vector $\gamma_k$:
\[
\gamma_k \triangleq B_{t,k}u_k^{i+1}+\beta_k^i-A_{t,k}x_k^{i+1}-e_{t,k}+\lambda_k^i
\]
thus we can represent the optimal solution $v_k^*$ with the optimal solution $z_{k+1}^*$: $v_k^*=\frac{1}{2}\Big(z_{k+1}^* + \gamma_k \Big).$ By substituting the above relationship into the update Eqn. \eqref{eqn_ADMM2_v_z_update}, we obtain an optimization problem involving only $z_{k+1}$:
\[
\begin{aligned}
    \min_{z_{k+1}\in\mathcal{X}}&~\frac{3}{4}z_{k+1}^\top z_{k+1}\\
    &~-z_{k+1}^\top\left(x_{k+1}^{i+1}+\theta_k^i+A_{t,k}x_{k}^{i+1}+e_{t,k}-\lambda_k^i+\frac{\gamma_k}{2}\right), 
\end{aligned}
\]
whose Hessian matrix is the identity matrix $I$, and which consequently yields the following closed-form solution,
\[
z_{k+1}^*=\mathrm{Proj}_{\mathcal{X}}\left(\frac{2(x_{k+1}^{i+1}+\theta_k^i+A_{t,k}x_{k}^{i+1}+e_{t,k}-\lambda_k^i)+\gamma_k}{3}\right),
\]
such as when the closed convex set $\mathcal{X}$ is box constraint: $\mathcal{X}=\{x:x_{\min}\leq x \leq x_{\max}\}$, halfspace: $\mathcal{X}=\{x:a^\top x\leq b\}$ (or affine hyperplane: $\mathcal{X}=\{x:a^\top x=b\}$), second-order cone $\mathcal{X}=\{(x,r):\|x\|_2\leq r\}$, $\ell_1$ ball: $\mathcal{X}=\{x:\|x\|_1\leq r \}$, see \cite[Ch. 6]{parikh2014proximal}. This completes the proof.
\end{proof}

\subsection{Closed-form ADMM Sub-steps via Velocity-based Formulation}\label{sec_velocity_formulation}
In the update Eqn. \eqref{eqn_ADMM2_u_x_update}, only $u_k$ is subject to the constraint  $\mathcal{U}$, whereas $x_{k+1}$  is unconstrained. Consequently, $u_k$ does not admit a closed-form solution, while $x_{k+1}$ does.

To make the update Eqn. \eqref{eqn_ADMM2_u_x_update} admitting a closed-form solution, one approach is to incorporate the control input $u_k$ into the state $x_{k+1}$ through the increased control input $\Delta u_k$.

\begin{theorem}\label{thm_2}
     By using the velocity-based formulation
\begin{equation}
    \begin{aligned}
        \Bar{x}_{k+1} =\Bar{A}_{t,k} \Bar{x}_k + \Bar{B}_{t,k}\Delta u_k + \Bar{e}_{t,k}
    \end{aligned}
\end{equation}
with $\Bar{x}_k\triangleq\mathrm{col}(x_{k}, u_{k-1})$, $\Bar{e}_{t,k}\triangleq\mathrm{col}(e_{t,k}, 0)$,
\begin{align*}
 \Bar{A}_{t,k}&\triangleq \left[\begin{array}{cc}
        A_{t,k} & B_{t,k} \\
        0 & I
    \end{array}\right],~\Bar{B}_{t,k}\triangleq \left[\begin{array}{c}
         B_{t,k}  \\
         I
    \end{array}\right]
\end{align*}
the corresponding splitting ADMM updates in \eqref{eqn_ADMM2_iter} all admit closed-form solutions.
\end{theorem}
\begin{proof}
After adopting the velocity-based MPC formulation, Problem \eqref{eqn_MPC} and applying splitting scheme \eqref{eqn_MPC_ADMM_2} can result in the following problem:
\begin{equation}\label{eqn_MPC_1}
    \begin{aligned}
    \min_{\Delta U,\Bar{X},V,Z}&~ \sum_{k=0}^{N-1} \frac{1}{2}\Bar{x}_{k+1}^\top \Bar{Q}\Bar{x}_{k+1}  - \bar{q}^\top \bar{x}_{k+1} + \Pi_{\mathcal{\Bar{X}}}(z_{k+1})\\  
        \text{s.t.}&~ \Bar{x}_{k+1} - z_{k+1}=0,~k=0,\cdots,N-1\\
    &~ \Bar{B}_{t,k}\Delta u_k - v_k=0,~k=0,\cdots,N-1\\
    &~ z_{k+1}-\Bar{A}_{t,k}\Bar{x}_k-v_k-\Bar{e}_{t,k}=0,~k=0,\cdots,N-1\\
        &~ \Bar{x}_0 = \left[\begin{array}{c}
             x(t)  \\
             u(t-1) 
        \end{array}\right]
    \end{aligned}
\end{equation}
where $\Bar{Q}\triangleq\diag(C^\top W_y C,W_u)$, $\bar{q}\triangleq\mathrm{col}(C^\top W_yr_y,W_ur_u)$, the set $\Bar{\mathcal{X}}\triangleq\mathcal{X}\times \mathcal{U}$, $\Delta U\triangleq\mathrm{col}(\Delta u_0,\cdots,\Delta u_{N-1})$, and $\Bar{X}\triangleq\mathrm{col}(\Bar{x}_1,\cdots,\Bar{x}_{N})$. 

Similar to the ADMM iteration \eqref{eqn_ADMM2_iter} and according to Theorem \ref{thm_v_z}, all updates $i$-th ADMM iteration for Problem \eqref{eqn_MPC_1} 
admit closed-form solutions, as given by
\begin{subequations}\label{eqn_ADMM_closed_iter}
{\footnotesize
\begin{empheq}[left=\empheqlbrace]{align}
&\textbf{for } k=0,\cdots, N-1~\textbf{(in parallel)}\label{eqn_ADMM_closed_iter_a} \\
    &\quad \Delta u_{k}^{i+1}=\argmin_{\Delta u_k}\frac{\rho}{2}\|\Bar{B}_{t,k}\Delta u_k-v_k^i+\beta_k^i\|_2^2 \nonumber\\
     &\quad\quad\quad\quad=\left(\Bar{B}_{t,k}^\top\Bar{B}_{t,k}\right)^{-1}\Bar{B}_{t,k}^\top\left(v_k^i-\beta_k^i\right) \nonumber \\
     &\quad \Bar{x}_{k+1}^{i+1}=\argmin_{\Bar{x}_{k+1}}~ \frac{1}{2}\Bar{x}_{k+1}^\top \Bar{Q}\Bar{x}_{k+1}- \bar{q}^\top \bar{x}_{k+1}\nonumber\\
     &\quad\quad\quad +\frac{\rho}{2}\|\Bar{x}_{k+1}-z_{k+1}^i+\theta_k^i\|_2^2 \nonumber\\
     &\quad\quad\quad + \frac{\rho}{2}\|z_{k+2}^i-\Bar{A}_{t,k}\Bar{x}_{k+1}-v_{k+1}^i-\Bar{e}_{t,k}+\lambda_{k+1}^i\|_2^2\nonumber\\
     &\quad\quad\quad~= H_k h_k\nonumber\\
    &\textbf{for } k=0,\cdots, N-1~\textbf{(in parallel)}\label{eqn_ADMM_closed_iter_b}\\
    &\quad \gamma_k = \Bar{B}_{t,k}\Delta u_k^{i+1}+\beta_k^i-\Bar{A}_{t,k}\Bar{x}_k^{i+1}-\Bar{e}_{t,k}+\lambda_k^i\nonumber\\
    &\quad z_{k+1}^{i+1}=\mathrm{Proj}_{\mathcal{\Bar{X}}}\left(\!\frac{2(\Bar{x}_{k+1}^{i+1}\!+\!\theta_k^i\!+\!\Bar{A}_{t,k}\Bar{x}_{k}^{i+1}\!+\!\Bar{e}_{t,k}\!-\!\lambda_k^i)\!+\!\gamma_k}{3}\!\right)\nonumber\\
    &\quad v_{i}^{i+1}=\frac{1}{2}\Big(z_{k+1}^{i+1}+\gamma_k\Big)\nonumber\\ 
   &\textbf{for } k=0,\cdots, N-1~\textbf{(in parallel)}\\
   &\quad\theta_k^{i+1}=\theta_k^i + \Bar{x}_{k+1}^{i+1}-z_{k+1}^{i+1}\nonumber \\
    &\quad\beta_k^{i+1}=\beta_k^{i} + \Bar{B}_{t,k}\Delta u_{k}^{i+1} - v_k^{i+1}\nonumber\\
    &\quad\lambda_k^{i+1}=\lambda_k^i + z_{k+1}^{i+1} -\Bar{A}_{t,k}\Bar{x}_{k}^{i+1} - v_{k}^{i+1}-\Bar{e}_{t,k}\nonumber
    \end{empheq}
}
\end{subequations}
where in Eqn. \eqref{eqn_ADMM_closed_iter_a} $\Bar{B}_{t,k}^\top\Bar{B}_{t,k}=B_{t,k}^\top B_{t,k}+I\succ0$ (ensuring the existence and uniqueness of the solution $\{\Delta u_k\}_{k=0}^{N-1}$)
and 
{
\footnotesize
\[
    \begin{aligned}
        H_k&\triangleq\left\{ \begin{array}{l}
                (\Bar{Q}+\rho I)^{-1},~\text{if }k=N-1\\
                (\Bar{Q}+\rho I+\rho\Bar{A}_{t,k+1}^\top \Bar{A}_{t,k+1})^{-1},~\text{else}           
            \end{array}\right. \\
        h_k&\triangleq\left\{\begin{array}{l}
            \bar{q}+\rho(z_{k+1}^i-\theta_k^i),~ \text{if }k=N-1\\
                 \bar{q}+\rho\Big(z_{k+1}^i-\theta_k^i+\Bar{A}_{t,k+1}^\top(z_{k+2}^i-v_{k+1}^i-\Bar{e}_{t,k}+\lambda_{k+1}^i)\Big),\text{else} 
            \end{array} \right.
   \end{aligned}
\]
}
where $H_k\succ0$ (as $H_k$ includes the term $\rho I$) for $k=0,\cdots,N_1$ (ensuring the existence and uniqueness of the solution $\{\bar{x}_{k+1}\}_{k=0}^{N-1}$). This completes the proof.
\end{proof}

\begin{remark}
\eqref{eqn_ADMM_closed_iter_a} involves the inverse of smaller matrices, such as  $\left(\Bar{B}_{t,k}^\top\Bar{B}_{t,k}\right)^{-1}$ and $(\Bar{Q}+\rho I+\rho\Bar{A}_{t,k+1}^\top \Bar{A}_{t,k+1})^{-1}$, whose computational cost is independent of the prediction horizon length $N$. Summarizing all the above, a \textbf{parallel-in-horizon} and \textbf{construction-free} ADMM algorithm tailored for MPC is shown in the \textbf{linear- and parallel-in-horizon} ADMM iteration \eqref{eqn_ADMM_closed_iter}. 
\end{remark}

\subsection{Residuals for Stopping Criteria}
For easy demonstration, we denote
\[
p\triangleq\mathrm{col}(\Delta U, \Bar{X}),~q\triangleq\mathrm{col}(V,Z),~s\triangleq\mathrm{col}(\Phi,\mathcal{B},\Lambda)
\]
Problem \eqref{eqn_MPC_1} can be rewritten as follows,
\begin{equation}\label{eqn_general_QP}
    \min~ f(p) + g(q),~ Fp + Gq = h 
\end{equation}
where the smooth function $f(p)$ and the indicator function $g(q)$ are both convex, as given by
\[
f(p)\triangleq \sum_{k=0}^{N-1} \frac{1}{2}\Bar{x}_{k+1}^\top \Bar{Q}\Bar{x}_{k+1}-\bar{q}^\top \bar{x},~g(q)\triangleq  \sum_{k=0}^{N-1}\Pi_{\mathcal{\Bar{X}}}(z_{k+1}),
\]
and the sparse matrices $F,G$ and the vector $h$ are corresponding to equality constraints of Problem \eqref{eqn_MPC_1}. As for the optimality conditions for the scaled ADMM iteration, the solution $p^*, q^*, s^*$ must satisfy the primal and dual feasibility conditions, as given by
\begin{subequations}\label{eqn_dual_opt}
    \begin{align}
        0= Fp^*+Gq^*-h&\quad\triangleright\textcolor{gray}{primal~ feasibility}\\
        0\in\partial f(p^*)+ \rho F^\top s^*&\quad\triangleright\textcolor{gray}{dual~ feasibility}\label{eqn_dual_opt_b}\\
        0\in\partial g(q^*)+ \rho G^\top s^*&\quad\triangleright\textcolor{gray}{dual~ feasibility}\label{eqn_dual_opt_c}
    \end{align}
\end{subequations}
After the $i$-th scaled-form ADMM iteration, $(p^{i+1},q^{i+1},s^{i+1})$ satisfies the following
\begin{subequations}
    \begin{align}
        0&\in\partial f(p^{i+1}) +\rho F^\top (s^{i}+Fp^{i+1}+Gq^{i}-h)\label{eqn_p_q_s_a}\\
        0&\in\partial g(q^{i+1}) + \rho G^\top(s^{i}+Fp^{i+1}+Gq^{i+1}-h)\label{eqn_p_q_s_b}    \\
        &=\partial g(q^{i+1}) + \rho G^\top s^{i+1}\quad\triangleright\textcolor{gray}{using ~\eqref{eqn_p_q_s_c}}\nonumber\\
        s^{i+1}& = s^{i}+Fp^{i+1}+Gq^{i+1}-h\label{eqn_p_q_s_c}
    \end{align}
\end{subequations}
which shows that the second dual optimality condition \eqref{eqn_dual_opt_c} is always satisfied. To obtain a useful expression for the condition \eqref{eqn_dual_opt_b}: $ 0\in\partial f(p^{i+1}) +\rho F^\top s^{i+1}$, substituting \eqref{eqn_p_q_s_c} into \eqref{eqn_p_q_s_a} yields
\[
\begin{aligned}
   0&\in\partial f(p^{i+1})\\
   &\quad+\rho F^\top\big(s^{i+1}-Fp^{i+1}-Gq^{i+1}+h +Fp^{i+1}+Gq^{i}-h \big)\\
   &=\partial f(p^{i+1}) + \rho F^\top s^{i+1} -\rho F^\top G(q^{i+1}-q^{i})\\
   &\quad (\text{as }f~\text{is smooth}, \in \text{can be replaced with} =)\\
   \Leftrightarrow&\quad\rho F^\top G(q^{i+1}-q^{i})= \partial f(p^{i+1}) + \rho F^\top s^{i+1} 
\end{aligned}
\]
\begin{lemma}
(see \cite[Appendix A]{boyd2011distributed}): Under Assumption in Remark \ref{remark_1}, the scaled-form ADMM iterates for Problem \eqref{eqn_general_QP} satisfy
\[
\sum_{i=0}^{\infty}\rho\|s^{i+1}-s^i\|_2^2 + \rho\|G(q^{i+1}-q^i)\|_2^2\leq V^0
\]
($V^0$: a positive constant) namely, $\|s^{i+1}-s^i\|_2^2\rightarrow0$ and $\|G(q^{i+1}-q^i)\|_2^2\rightarrow0$, which implies the primal residual ($Fp^{i+1}+Gq^{i+1}-h$) and dual residual ($\rho F^\top G(q^{i+1}-q^{i})$) both converge to zeros.
\end{lemma}
Furthermore, we introduce
\[
c^{i+1} \triangleq\rho\|s^{i+1}-s^i\|_2^2 + \rho\|G(q^{i+1}-q^i)\|_2^2
\]
which denotes the sum of the Euclidean norms of the residuals. \cite{he2015non} proves that $c^{i+1}$ decreases monotonically (noting that we use the scaled form, differing from the standard form in \cite{he2015non}), which is key to proving convergence for the restart method present in the following Subsection.
\begin{lemma}
(see \cite[Theorem 4.1]{he2015non})
    Under Assumption in Remark \ref{remark_1}, the scaled-form ADMM iterates for Problem \eqref{eqn_general_QP} satisfy
    \[
        c^{i+1}\leq c^i.
    \]
\end{lemma}

\begin{remark}
Thus, a reasonable termination criterion for detecting optimality is 
$
c^{i+1}\leq\epsilon.
 $
Specifically, in the proposed \textbf{linear-} and \textbf{parallel-in-horizon} ADMM-based MPC algorithm \eqref{eqn_ADMM_closed_iter}, the instantiated $c^{i+1}$ are given by
\[
    \begin{aligned}
        c^{i+1}&=\rho\sum_{k=0}^{N-1}\|\theta_k^{i+1}-\theta^i\|_2^2+\|\beta_k^{i+1}-\beta^i\|_2^2+\|\lambda_k^{i+1}-\lambda^i\|_2^2 \\ &\quad\quad+\|z^{i+1}_{k+1}-\hat{z}^i_k\|_2^2+\|v_k^{i+1}-v^i_k\|_2^2\\
        &\quad\quad+\|(z_{k+1}^{i+1}- v_{k}^{i+1}) - (z_{k+1}^i -v_{k}^i)\|_2^2\\
        &\leq \epsilon 
    \end{aligned}
\]
\end{remark}

\subsection{Accelerated ADMM with Restart}
Although the per-iteration complexity of the proposed \textbf{linear-} and \textbf{parallel-in-horizon} ADMM-based MPC algorithm \eqref{eqn_ADMM_closed_iter} scales favorably with problem size, the total number of iterations typically increases for high-dimensional problems due to poorer operator conditioning and stronger variable coupling, especially for large horizon lengths $N$. To mitigate this, accelerated ADMM variants, including Nesterov-type accelerated schemes with established global convergence guarantees for strongly convex problems, can yield a significant speed up. However, in our novel-splitting Problem \eqref{eqn_MPC_1} (or \eqref{eqn_general_QP}), the objective functions $f(p)$ and $g(q)$ are only convex.

Ref. \cite{goldstein2014fast} introduces an accelerated ADMM scheme for convex problems. In contrast to the strongly convex setting, no global convergence rate is derived; instead, convergence is enforced through the use of Nesterov’s acceleration together with an appropriate restart strategy. The pseudo-code of the accelerated ADMM with restart is shown in \eqref{eqn_faster_pseudo_code}. Nevertheless, the empirical behavior of  \eqref{eqn_faster_pseudo_code} is superior to that of the original ADMM, even in the case of strongly convex functions. From \cite[Sec 4.3]{goldstein2014fast}, the accelerated ADMM with restart scheme in \eqref{eqn_faster_pseudo_code} for Problem \eqref{eqn_general_QP} is convergent.

\begin{lemma}
    (see \cite[Thm. 3]{goldstein2014fast}): For convex functions $f(p)$ and $g(q)$, the accelerated ADMM with restart scheme in \eqref{eqn_faster_pseudo_code} converges in the sense that $\lim_{i\rightarrow\infty} c^i\rightarrow0$.
\end{lemma}

\begin{subequations}\label{eqn_faster_pseudo_code}
{\small
\begin{empheq}[left=\empheqlbrace]{align}
    &\textbf{Pseudo code of Accelerated ADMM with restart:}\nonumber\\
    &\textbf{Given } p^0,\hat{q}^{0}=q^0, ~\hat{\xi}^{0}=\xi^0,\alpha_0=1,\eta=0.999,c^0=1\nonumber\\
        &\textbf{for } i=1,2\cdots,\nonumber\\
        &\quad p^{i}=\argmin_{p}~f(p)+\frac{\rho}{2}\|Fp+G\hat{q}^i-h+\hat{\xi}^i\|_2^2\\
        &\quad q^{i}=\argmin_{q}~g(q)+\frac{\rho}{2}\|Fp^{i+1}+Gq-h+\hat{\xi}^i\|_2^2\\
        &\quad \xi^{i}=\hat{\xi}^i + Fp^{i+1} + Gq^{i+1} - h\\
        &\quad c^{i+1} =\rho\|\xi^{i}-\hat{\xi}^i \|_2^2 + \rho\|G( q^{i} -\hat{q}^i)\|_2^2\\
        &\quad \textbf{if } c^{i+1}\leq\eta c^i\nonumber\\
        &\quad\quad \alpha_{i+1}=\frac{1+\sqrt{1+4\alpha_i^2}}{2}\nonumber\\
        &\quad\quad \hat{q}^{i+1}=q^{i} + \frac{\alpha_i-1}{\alpha_{i+1}}(q^{i}-q^{i-1})\nonumber\\
        &\quad\quad \hat{\xi}^{i+1}=\xi^{i}+ \frac{\alpha_i-1}{\alpha_{i+1}}(\xi^{i}-\xi^{i-1})\nonumber\\
        &\quad \textbf{else}\nonumber\\
        &\quad\quad \alpha^{i+1}=1, \hat{q}^{i+1} = q^{i-1},\hat{\xi}^{i+1} = \xi^{i-1}\nonumber
    \end{empheq}
}
\end{subequations}

Applying accelerated ADMM with restart \eqref{eqn_faster_pseudo_code} to the \textit{parallel-in-horizon} and \textit{construction-free} iterations in \eqref{eqn_ADMM_closed_iter}, together with a stopping criterion, yields the final $\pi$MPC algorithm in Algorithm~\ref{alg_ADMM_MPC}.

\begin{algorithm}[!htbp]
    \caption{$\pi$MPC: parallel-in-horizon and construction-free ADMM algorithm for MPC} \label{alg_ADMM_MPC}
    \textbf{Input}: Given $\Delta u_k^0$, $\Bar{x}_{k+1}^0$, $\hat{v}_{k}^0=v_{k}^0$, $\hat{z}_{k+1}^0=z_{k+1}^0$, $\hat{\theta}_{k}^0=\theta_{k}^0$, $\hat{\beta}_{k}^0=\beta_{k}^0$, $\hat{\lambda}_{k}^0=\lambda_{k}^0$ with $k=0,\cdots,N-1$ (such as from warm-start); the parameter $\rho$, a desired optimal level $\epsilon$, and the maximum number of iterations $K_{\max}$, the current states: $\Bar{x}_0=\mathrm{col}(x(t),0)$; $\alpha_0=1, \eta=0.999,c^0=1$.
    \vspace*{.1cm}\hrule\vspace*{.1cm}
    \textbf{cache} $\left\{J_k\triangleq\left(\Bar{B}_{t,k}^\top\Bar{B}_{t,k}\right)^{-1}\Bar{B}_{t,k}^\top\right\}_{k=0}^{N-1}$ and $\left\{H_k\right\}_{k=0}^{N-1}$;\\
    \textbf{for} $i=0,1, 2,\cdots{},K_{\max}-1$ \textbf{do}
    \begin{enumerate}[label*=\arabic*., ref=\theenumi{}]
        \item \textbf{for} $k=0,1,\cdots, N-1$ \textbf{(in parallel)}
        \[
        \begin{aligned}
            \Delta u_{k}^{i+1}&=J_k\left(\hat{v}_k^i-\hat{\beta}_k^i\right),\\
            \Bar{x}_{k+1}^{i+1}&= H_kh_k    
        \end{aligned}
        \]
        with
        {\footnotesize
        \[
         h_k\triangleq\left\{\begin{array}{l}
            \!\!\bar{q}\!+\!\rho(\hat{z}_{k+1}^i\!-\!\hat{\theta}_k^i),~ \text{if }k=N-1\\
            \!\!\bar{q}\!+\!\rho\Big(\hat{z}_{k+1}^i\!-\!\hat{\theta}_k^i\!+\!\Bar{A}_{t,k}^\top(\hat{z}_{k+2}^i\!-\!\hat{v}_{k+1}^i\!-\!\Bar{e}_{t,k}\!+\!\hat{\lambda}_{k+1}^i)\Big),\text{else} 
            \end{array} \right.
        \]
        }
        \item \textbf{for} $k=0,1,\cdots, N-1$ \textbf{(in parallel)}
        {\small
        \[
        \begin{aligned}
        \gamma_k &= \Bar{B}_{t,k}\Delta u_k^{i+1}+\hat{\beta}_k^i-\Bar{A}_{t,k}\Bar{x}_k^{i+1}-\Bar{e}_{t,k}+\hat{\lambda}_k^i\\
    z_{k+1}^{i+1}&=\mathrm{Proj}_{\mathcal{\Bar{X}}}\left(\!\frac{2(\Bar{x}_{k+1}^{i+1}\!+\!\hat{\theta}_k^i\!+\!\Bar{A}_{t,k}\Bar{x}_{k}^{i+1}\!+\!\Bar{e}_{t,k}\!-\!\hat{\lambda}_k^i)\!+\!\gamma_k}{3}\!\right)\\
    v_{i}^{i+1}&=\frac{1}{2}\Big(z_{k+1}^{i+1}+\gamma_k\Big)
        \end{aligned}
        \]
        }
        \item \textbf{for} $k=0,1,\cdots, N-1$ \textbf{(in parallel)}
        \[
            \begin{aligned}
            \theta_k^{i+1}&=\hat{\theta}_k^i + \Bar{x}_{k+1}^{i+1}-z_{k+1}^{i+1} \\
            \beta_k^{i+1}&=\hat{\beta}_k^{i} + \Bar{B}_{t,k}\Delta u_{k}^{i+1} - v_k^{i+1}\\
            \lambda_k^{i+1}&=\hat{\lambda}_k^i + z_{k+1}^{i+1} -\Bar{A}_{t,k}\Bar{x}_{k}^{i+1} - v_{k}^{i+1}-\Bar{e}_{t,k}
            \end{aligned}
        \]
    \item \textbf{compute}
    \[
    \begin{aligned}
        c^{i+1}\leftarrow\rho\sum_{k=0}^{N-1}&\|\theta_k^{i+1}-\hat{\theta}^i\|_2^2+\|\beta_k^{i+1}-\hat{\beta}^i\|_2^2+\|\lambda_k^{i+1}-\hat{\lambda}^i\|_2^2 \\ +&\|z^{i+1}_{k+1}-\hat{z}^i_k\|_2^2+\|v_k^{i+1}-\hat{v}^i_k\|_2^2\\
        +&\|(z_{k+1}^{i+1}- v_{k}^{i+1}) - (\hat{z}_{k+1}^i -\hat{v}_{k}^i)\|_2^2    
    \end{aligned}
    \]
    \item \textbf{if} $c^{i+1}\leq\epsilon$, \textbf{break};\\
    \item \textbf{if} $c^{i+1}\leq\eta c^i$, 
    {\footnotesize
    \[
    \begin{aligned}
        &\alpha_{i+1}=0.5+0.5\sqrt{1+4\alpha_i^2}\\
        &(\hat{z}_{k+1}^{i+1},\hat{v}_k^{i+1},\hat{\theta}_k^{i+1},\hat{\beta}_k^{i+1},\hat{\lambda}_k^{i+1})= (z_{k+1}^{i+1},v_k^{i+1},\theta_k^{i+1},\beta_k^{i+1},\lambda_k^{i+1})\\
        &+\frac{\alpha_i-1}{\alpha_{i+1}}\Big(z_{k+1}^{i+1},v_k^{i+1},\theta_k^{i+1},\beta_k^{i+1},\lambda_k^{i+1}) - (z_{k+1}^{i},v_k^{i},\theta_k^{i},\beta_k^{i},\lambda_k^{i})\Big)\\
        &\text{with } k=0,\cdots,N-1 \textbf{ in parallel}
    \end{aligned}
    \]
    }
    \item \textbf{else} 
    {\footnotesize
    \[
    \begin{aligned}
    &\alpha_{i+1}=1 \\
    &(\hat{z}_{k+1}^{i+1},\hat{v}_k^{i+1},\hat{\theta}_k^{i+1},\hat{\beta}_k^{i+1},\hat{\lambda}_k^{i+1})= (z_{k+1}^{i+1},v_k^{i+1},\theta_k^{i+1},\beta_k^{i+1},\lambda_k^{i+1})\\
    &\text{with } k=0,\cdots,N-1 \textbf{ in parallel}
    \end{aligned}
    \]
    }
    \end{enumerate}
    \textbf{end}\vspace*{.1cm}\hrule\vspace*{.1cm}
    \textbf{Output:} $(\Delta u_k^{i+1},~\Bar{x}_{k+1}^{i+1})$, $(v_k^{i+1},~z_{k+1}^{i+1})$, $(\theta_k^{i+1},~\beta_k^{i+1},~\lambda_k^{i+1})$ with $k=0,\cdots,N-1$ for warm-start in $(t+1)$-th sampling time. Applying the first control input: $u_0=\Delta u_0$, in MPC.
\end{algorithm}

\section{Numerical Example}
This section benchmarks $\pi$MPC against a QP-ADMM (implemented by the authors) OSQP~\cite{SBGBB20} (a state-of-the-art implementation of QP-ADMM), and structure-exploiting Riccati-ADMM (implemented by the authors). First, an ill-conditioned AFTI-16 MPC example is used to evaluate whether the iteration count of $\pi$MPC increases. Second, randomly generated linear MPC problems demonstrate that the per-iteration cost of $\pi$MPC is nearly independent of problem size on GPU. Finally, a nonlinear CSTR example illustrates that $\pi$MPC is so natural to support fast NMPC applications. 

All simulations are conducted using Julia on a desktop with an AMD Ryzen 9 5900X CPU and an NVIDIA GeForce RTX 3080 GPU. 
Codes are available at \url{https://github.com/SOLARIS-JHU/piMPC}.

\subsection{AFTI-16 Aircraft System}
The open-loop unstable linearized model of the AFTI-16 aircraft is from \cite{557577} as shown below,
{
\footnotesize
\begin{align*}
    \dot{x}= & \begin{bmatrix}
    -0.0151 & -60.5651 & 0 & -32.174 \\
    -0.0001 & -1.3411 & 0.9929 & 0 \\
    0.00018 & 43.2541 & -0.86939 & 0 \\
    0 & 0 & 1 & 0
    \end{bmatrix} x  \\
    & +\begin{bmatrix}
    -2.516 & -13.136 \\
    -0.1689 & -0.2514 \\
    -17.251 & -1.5766 \\
    0 & 0
    \end{bmatrix} u \\
    y = & \begin{bmatrix}
        0 & 1 & 0 & 0 \\
        0 & 0 & 0 & 1
    \end{bmatrix}x
\end{align*}
}

The model is discretized with a sampling period of $0.05$ s using a zero-order hold. 
Input constraints are imposed as $|u_i| \leq 25^\circ$, $i = 1, 2$, while output constraints are $-0.5 \leq y_1 \leq 0.5$ and $-100 \leq y_2 \leq 100$. 
The objective is to track reference $r_2$ with pitch angle $y_2$.
The MPC weighting matrices are set to $W_y = \mathrm{diag}(100, 100)$, $W_u = 0$, and $W_{\Delta u} = \mathrm{diag}(0.1, 0.1)$, with prediction horizon $N = 5$.

Fig.~\ref{fig:AFTI16_traj} shows the closed-loop tracking performance, where $y_2$ tracks the reference signal and all output and input constraints are satisfied. 
Fig.~\ref{fig:AFTI16_res} compares the number of iterations required to reach $10^{-6}$ residual at each MPC step. Averaged over 200 MPC steps, the mean number of iterations per step is  \textbf{1326} for OSQP, \textbf{2030} for $\pi$MPC, \textbf{2551} for Riccati-ADMM, and \textbf{5207} for QP-ADMM. These results indicate that the parallel-in-horizon and construction-free $\pi$MPC achieves iteration-level convergence performance comparable to state-of-the-art ADMM variants.

\begin{figure}[htbp]
\centering
\vspace{-1em}
\includegraphics[width=\linewidth]{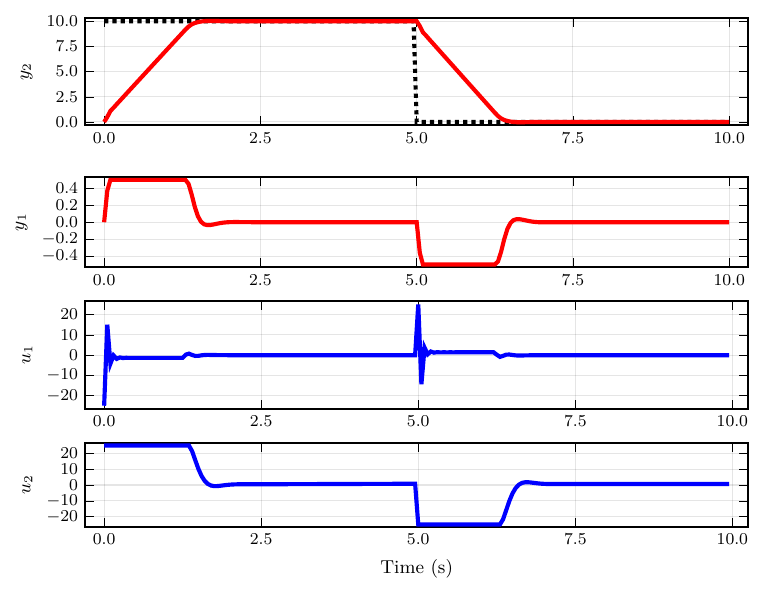}
\vspace{-2em}
\caption{AFTI-16 closed-loop trajectory tracking performance}
\label{fig:AFTI16_traj}
\end{figure}

\begin{figure}[htbp]
\centering
\vspace{-1em}
\includegraphics[width=\linewidth]{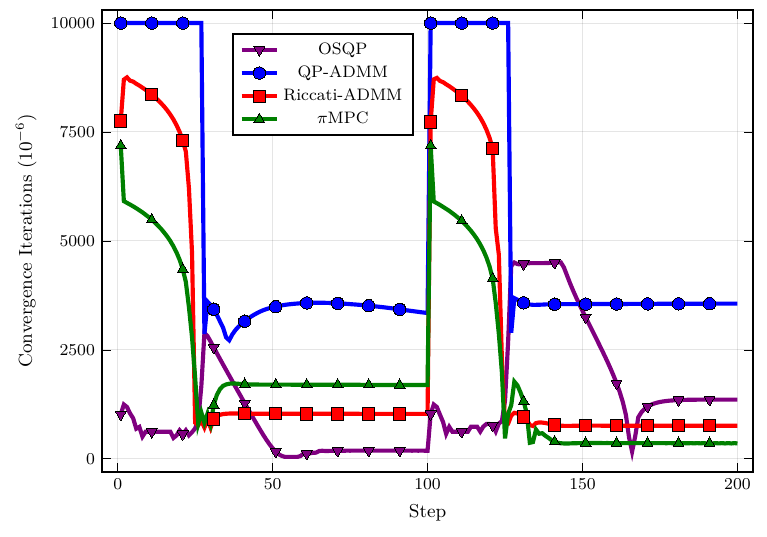}
\vspace{-2em}
\caption{Convergence iterations to $10^{-6}$ residual at each MPC step for AFTI-16 (capped at 10,000)}
\label{fig:AFTI16_res}
\end{figure}

\subsection{Randomly Generated System}
To benchmark scalability, we use randomly generated systems~\cite{Wang2010FastMP} with dynamics $x_{t+1} = Ax_t + Bu_t + w_t$, where $A$ and $B$ have entries drawn from $\mathcal{N}(0, 1)$ and $A$ is scaled to unit spectral radius. State disturbances $w_t$ have entries uniformly distributed on $[-0.3, 0.3]$.
We vary the number of states $n$, inputs $m$, and prediction horizon $N$ to assess performance across different problem scales.
Box constraints are imposed with $|x_i| \leq 5$ for states and $|u_i| \leq 0.1$ for inputs, with state weighting $W_y = I_n$, input weighting $W_u = I_m$, and $W_{\Delta u} = 0$.

Table~\ref{tab:scalability} and Fig.~\ref{fig:scalability} present the scalability comparison across problem scales based on average per-iteration time over 200 iterations. 
OSQP (a state-of-the-art implementation of QP-ADMM) is the fastest on small-scale problems but exhibits increasing time growth with problem scale, while QP-ADMM shows more severe scaling issues.
Riccati-ADMM scales better than QP-ADMM but still shows notable time growth on large-scale problems.
$\pi$MPC on CPU exhibits moderate and consistent scaling, though slower than OSQP on small-scale problems.
$\pi$MPC on GPU maintains nearly constant computation time across all scales, achieving substantial speedup on large-scale problems.

\begin{table}[htbp]
\centering
\vspace{-1em}
\caption{Scalability Benchmark on Randomly Generated Systems}
\label{tab:scalability}
\begin{tabular}{rrrrrrrr}
\toprule
$n$ & $m$ & $N$ & OSQP & \makecell{QP-\\ADMM} & \makecell{Riccati-\\ADMM} & \makecell{$\pi\mathrm{MPC}$\\(CPU)} & \makecell{$\pi\mathrm{MPC}$\\(GPU)} \\
\midrule
 10 &  3 &  20 &     \textbf{0.02} &     0.09 &   0.34 &   0.15 &   0.53 \\
 10 &  3 &  50 &     \textbf{0.06} &     0.50 &   0.99 &   0.31 &   0.55 \\
 10 &  3 & 100 &     \textbf{0.11} &     4.31 &   1.93 &   0.58 &   0.54 \\
 10 &  3 & 200 &     \textbf{0.23} &    59.44 &   3.94 &   1.12 &   0.55 \\
 10 &  3 & 500 &     0.59 &   420.39 &  10.63 &   2.75 &   \textbf{0.56} \\
\midrule
 20 &  6 &  20 &     \textbf{0.07} &     0.40 &   0.67 &   0.17 &   0.56 \\
 20 &  6 &  50 &     \textbf{0.18} &     4.74 &   1.45 &   0.35 &   0.56 \\
 20 &  6 & 100 &     \textbf{0.35} &    59.00 &   2.84 &   0.68 &   0.55 \\
 20 &  6 & 200 &     0.72 &   266.25 &   6.04 &   1.38 &   \textbf{0.55} \\
 20 &  6 & 500 &     1.85 &  1639.62 &  14.63 &   3.52 &   \textbf{0.56} \\
\midrule
 30 & 10 &  20 &     \textbf{0.15} &     0.74 &   1.10 &   0.18 &   0.54 \\
 30 & 10 &  50 &     \textbf{0.37} &    30.76 &   2.66 &   0.38 &   0.54 \\
 30 & 10 & 100 &     0.74 &   158.66 &   5.63 &   0.73 &   \textbf{0.58} \\
 30 & 10 & 200 &     1.56 &   627.56 &  10.59 &   1.61 &   \textbf{0.56} \\
 30 & 10 & 500 &     5.78 &      --- &  26.91 &   3.90 &   \textbf{0.56} \\
\midrule
 50 & 15 &  20 &     \textbf{0.34} &     5.33 &   3.52 &   0.23 &   0.55 \\
 50 & 15 &  50 &     0.94 &    99.96 &   8.95 &   \textbf{0.47} &   0.56 \\
 50 & 15 & 100 &     2.09 &   414.49 &  18.08 &   1.24 &   \textbf{0.58} \\
 50 & 15 & 200 &     7.35 &  1653.79 &  36.67 &   2.08 &   \textbf{0.57} \\
 50 & 15 & 500 &    24.10 &      --- &  86.04 &   4.99 &   0.59 \\
\midrule
100 & 30 &  20 &     1.49 &    67.60 &  14.74 &   \textbf{0.35} &   0.56 \\
100 & 30 &  50 &     6.51 &   416.92 &  39.59 &   0.96 &   \textbf{0.58} \\
100 & 30 & 100 &    17.49 &  1650.60 &  70.37 &   1.72 &   \textbf{0.57} \\
100 & 30 & 200 &    38.46 &      --- & 156.00 &   3.15 &   \textbf{0.60} \\
100 & 30 & 500 &    99.36 &      --- & 425.69 &   8.14 &   \textbf{0.62} \\
\bottomrule
\end{tabular}
\begin{tablenotes}
\small
\item \textit{Note}: ``---'' indicates out-of-memory failure. All times are per-iteration computation times in \textbf{milliseconds}.
\end{tablenotes}
\vspace{-1em}
\end{table}

\begin{figure}[htbp]
\centering
\vspace{-0.8em}
\includegraphics[width=\linewidth]{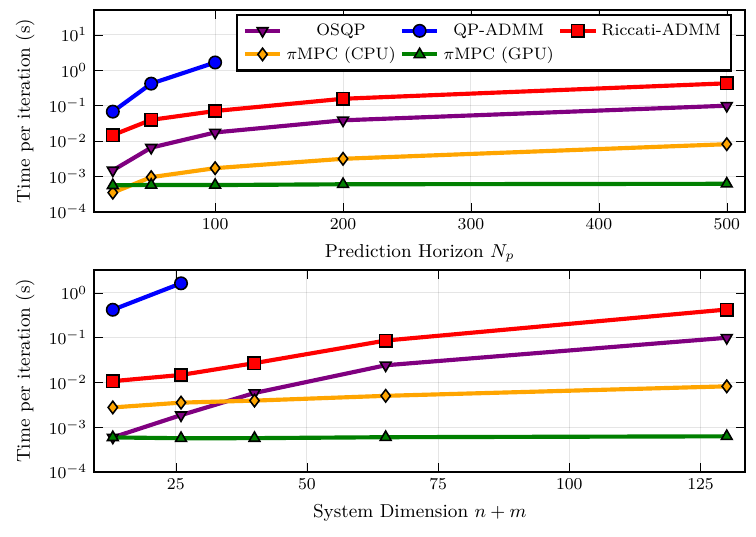}
\vspace{-2em}
\caption{Per-iteration computation time as a function of (a) prediction horizon, where system dimension is $(n, m) = (100, 30)$ and (b) system dimension, where the prediction horizon is $N = 500$}
\label{fig:scalability}
\end{figure}

\subsection{Nonlinear CSTR System}
To assess performance on nonlinear systems, we consider a CSTR, described by
{
\footnotesize
\[
\begin{aligned}
\frac{dC_A}{dt} &= C_{A,i} - C_A - k_0 e^{-EaR/T} C_A, \\
\frac{dT}{dt} &= T_i + 0.3T_c - 1.3T + 11.92k_0 e^{-EaR/T} C_A, \\
y &= C_A,    
\end{aligned}
\]
}

where $C_A$ is the concentration (kgmol/m$^3$), $T$ is the reactor temperature (K), and $T_c$ is the coolant temperature. 
The inlet concentration is $C_{A,i} = 10.0$ kgmol/m$^3$, and the inlet temperature disturbance is $T_i = 298.15 + 5\sin(0.05t)$ K. 
Model constants are $k_0 = 34930800$ and $EaR = 5963.6$ K. 
The objective is to drive the system from a low-conversion state ($C_A = 8.57$ kgmol/m$^3$, $T = 311$ K) to high conversion ($C_A = 2$ kgmol/m$^3$) while rejecting the inlet temperature disturbance.
The nonlinearity is handled via online linearization, where the system is linearized around the current state at each time step using forward Euler discretization with $T_s = 0.5$ min.
The MPC uses a prediction horizon of $N = 5000$ with weighting matrices $W_y = 1$, $W_u = 0$, and $W_{\Delta u} = 0.1$, where the long prediction horizon demonstrates the scalability benefits of GPU-accelerated $\pi$MPC.

Table~\ref{tab:cstr_stats} summarizes the computation time over 400 MPC steps with 1000 iterations each, demonstrating that $\pi$MPC outperforms OSQP. OSQP incurs a per-step QP matrix construction cost comparable to its solver runtime, whereas $\pi$MPC is construction-free. Fig.~\ref{fig:CSTR_traj} shows the closed-loop tracking performance under inlet temperature disturbance, with both methods achieving consistent results.

\begin{table}[htbp]
\centering
\caption{CSTR Computation Time Statistics}
\label{tab:cstr_stats}
\begin{tabular}{lcc}
\toprule
 & OSQP$^*$ & $\pi\mathrm{MPC}^\dagger$ \\
\midrule
Median & 1.891 (1.073) & 0.526 \\
Max & 1.970 (1.120) & 0.610 \\
\bottomrule
\end{tabular}
\begin{tablenotes}
\small
\item \textit{Note}: $^*$Total time (construction time in parentheses); $^\dagger$Solution time only (construction-free). All times are in seconds per MPC step over 400 steps with 1000 iterations each and prediction horizon $N=5000$.
\end{tablenotes}
\end{table}

\begin{figure}[htbp]
\centering
\vspace{-1em}
\includegraphics[width=\linewidth]{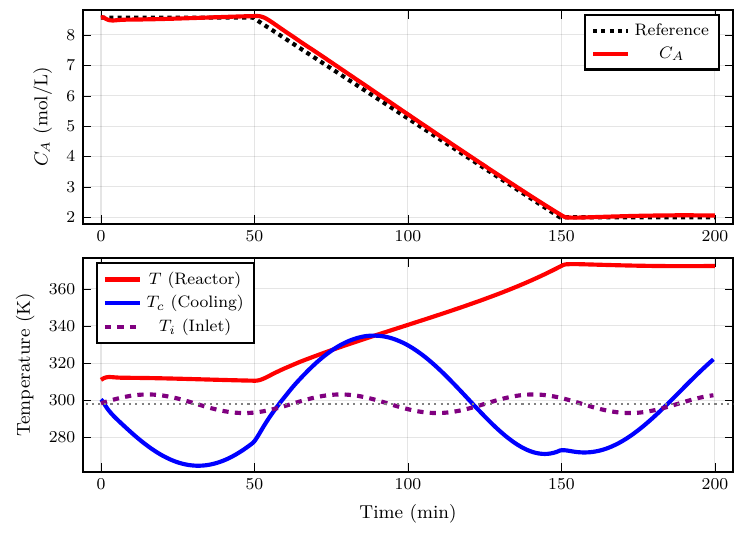}
\vspace{-2em}
\caption{CSTR trajectory tracking with inlet temperature disturbance}
\label{fig:CSTR_traj}
\end{figure}

\section{Conclusion}
This technical note presented $\pi$MPC, a novel \textit{parallel-in-horizon} and \textit{construction-free} ADMM-based algorithm for NMPC problems. By introducing a novel variable-splitting scheme together with a velocity-based formulation, $\pi$MPC enables horizon-wise parallel execution while operating directly on system matrices, thereby avoiding explicit MPC-to-QP construction. $\pi$MPC admits closed-form iteration updates with linear and parallel computational complexity in the prediction horizon, making it particularly suitable for long-horizon and real-time NMPC applications on embedded controller platforms. Future work will focus on developing a differentiable version of $\pi$MPC that can be used as a layer within an end-to-end deep neural network.

\bibliographystyle{IEEEtran}
\bibliography{ref_piMPC} 
\end{document}